\documentclass[12pt, reqno]{amsart}

\usepackage[utf8]{inputenc}

\usepackage{amssymb, amsmath, amsthm, esint}

\usepackage{latexsym}

\usepackage{graphicx}

\usepackage{color}

\newtheorem{defi}{Definition}[section]
\newtheorem{thm}[defi]{Theorem}
\newtheorem{ex}[defi]{Example}
\newtheorem{rem}[defi]{Remark}
\newtheorem{prop}[defi]{Proposition}
\newtheorem{lemma}[defi]{Lemma}

\newtheorem{obs}[defi]{Observation}

\newcommand{\Sp}{\mathbb S}

\newcommand{\D}{\mathbb{D}}
\newcommand{\R}{\mathbb R}
\newcommand{\N}{\mathbb N}

\title[Steklov eigenvalues for
  prescribed boundary]{Steklov eigenvalues of
  submanifolds with prescribed boundary in Euclidean space}
\author{Bruno Colbois}
\address{Universit\'e de Neuch\^atel, Institut de Math\'ematiques, Rue
  Emile-Argand 11, CH-2000 Neuch\^atel, Switzerland}
\email{bruno.colbois@unine.ch}
\author{Alexandre Girouard}
\address{Département de mathématiques et de statistique,
Pavillon Alexandre-Vachon, Université Laval,
Québec, QC, G1V 0A6, Canada}
\email{Alexandre.Girouard@ulaval.ca}
\author{Katie Gittins}
\address{Max Planck Institute for Mathematics, Vivatsgasse 7, 53111 Bonn, Germany}
\email{gittinsk@mpim-bonn.mpg.de}
\date{\today.}

\begin{document}
\begin{abstract}
  We obtain upper and lower bounds for Steklov eigenvalues of
  submanifolds with prescribed boundary in Euclidean space. A very general upper bound is
  proved, which depends only on the geometry of the fixed boundary and
  on the measure of the interior. Sharp lower bounds are given for
  hypersurfaces of revolution with connected boundary: we prove
  that each eigenvalue is uniquely minimized by the ball. We
  also observe that each surface of revolution with connected boundary
  is isospectral to the disk.
\end{abstract}
\maketitle

\section{Introduction}
The Steklov eigenvalues of a smooth, compact, connected Riemannian
manifold $(M,g)$ of dimension $n\geq 2$ with boundary $\Sigma$ are the real
numbers $\sigma$ for which there exists a nonzero harmonic function
$f:M\rightarrow\R$ which satisfies $\partial_\nu f=\sigma f$ on the
boundary $\Sigma$. Here and in what follows,
$\partial_\nu$ is the outward normal derivative on $\Sigma$. The Steklov
eigenvalues form a discrete sequence
$0=\sigma_0<\sigma_1\leq\sigma_2\leq\cdots\nearrow\infty,$
where each eigenvalue is repeated according to its multiplicity.

Several recent papers have investigated the
effects
of the geometry near the boundary on the Steklov eigenvalues,
for instance \cite{ceg2, ceg3, GPPS, prostu, YangYu}. In
particular, in \cite{ceg3} the authors investigated the situation where
the Riemannian metric $g$ is prescribed on the boundary $\Sigma$. On
any  manifold of dimension $n\geq 3$ they proved that
conformal perturbations $\omega^2g$ with $\omega\equiv 1$ on $\Sigma$
can make $\sigma_1$ arbitrarily large, as well as arbitrarily
small. Hence
the prescription of $g$ on the boundary does
not constrain the spectrum very much.

In the present paper, we study upper and lower bounds for Steklov eigenvalues in an even more rigid situation,
where a closed (possibly not connected) submanifold
$\Sigma\subset\R^{m}$ of dimension $n-1$ is given, and
where we consider the Steklov spectrum of submanifolds $M\subset\R^{m}$ with prescribed boundary
$\Sigma$.

Regarding upper bounds in this setting, our first result shows that it
is possible to control each eigenvalue $\sigma_k$ by simply
controlling the volume of $M$.
\begin{thm}\label{thm:verygeneralupperbound}
  Let $\Sigma$ be a fixed $(n-1)$-dimensional compact, smooth
  submanifold in $\R^{m}$. There exists a constant
  $A$ depending on $\Sigma$ such that any
  compact $n$-dimensional
  submanifold $M$ of $\R^m$ with boundary $\Sigma$ satisfies
 \begin{equation*}
   \sigma_k(M) \le A \vert M\vert k^{2/(n-1)}.
 \end{equation*}
\end{thm}
\begin{rem}
  The dependance of the constant $A$ on the geometry of $\Sigma$ is
  described in Section~\ref{section:proofverygeneral}. It involves the volume of $\Sigma$, a lower
  bound on its Ricci curvature and an upper bound on its diameter. It
  also depends on the number of connected components of $\Sigma$ as
  well as their maximal distortion (See formula \eqref{def:distortion} for the
  definition of distortion).
\end{rem}
\begin{rem}
  In the specific case of $\sigma_1$, a similar result was proved by Ilias and
  Makhoul~\cite[Theorem 1.2]{IliasMakhoul}. Their bounds are in terms of the $r$-th mean
  curvatures $H_r$ of
  $\Sigma$ in the ambient space $\R^m$, and in terms of the measure $|M|$.
\end{rem}
\begin{rem}
  In the abstract Riemannian setting, there exist situations
  where $|\Sigma|=1$, $|M|\leq 2$ and $\sigma_1$ is arbitrarily large
  (See \cite[Proposition 6.2]{ceg2}). By the Nash isometric embedding theorem, we can
  realize these examples as submanifolds in $\R^m$. The constant $A$
  must therefore depend on $\Sigma$.
\end{rem}

In the situation where $\Sigma$ bounds a domain $\Omega$ in a linear subspace of dimension $n$, it is possible to
obtain an upper bound which depends only on $|\Omega|$, $|M|$ and the dimension. In
particular, it does not involve the Ricci curvature and diameter of $\Sigma$.
\begin{thm}\label{thm:generalupperbound}
  Let $\Sigma$ be an $(n-1)$-dimensional, connected, smooth
  hypersurface in $\R^n \times \{0\} \subset \R^m$ and $\Omega \subset \R^{n} \times \{0\}$ denote the domain with
  boundary $ \Sigma=\partial \Omega$. Let $M \subset \R^{m}$
  be an $n$-dimensional hypersurface  with boundary $\partial M =
  \Sigma$. Then, for each $k \ge 1$,
  \begin{equation}\label{e1.1}
    \sigma_k(M) \leq A(n)
    \frac{\vert M \vert}{\vert \Omega \vert^{(n+1)/n}} k^{2/(n-1)},
  \end{equation}
  where $A(n)$ is a constant depending only on the dimension $n$.
\end{thm}

\begin{rem}
  The control of the volume $|\Omega|$ in \eqref{e1.1} is
  necessary. Indeed, it was proved in \cite[Section 5]{CDE} that
  there exists a sequence of compact hypersurfaces $\Sigma_\ell\subset\R^{n}\times\{0\}$
  of volume $|\Sigma_\ell|=1$ such that
  $\lambda_1(\Sigma_\ell)\to\infty$. It follows from \cite{ceg1} that
  $\Sigma_\ell$  bounds a domain $\Omega_\ell$ with $|\Omega_\ell|\to 0$ as $\ell\to\infty$.
  Now, consider $M_\ell\subset\R^{n+1}$ an $n$-dimensional manifold with
  connected boundary $\partial M_\ell=\Sigma_\ell$, uniformly bounded
  volume $|M_\ell|$ and such that a neighborhood of the
  boundary is isometric to the cylinder $\Sigma_\ell\times [0,1)$. Then,
  from the Dirichlet-Neumann bracketing \eqref{ineq:compmixed} and Lemma
  \ref{lemma:EigenvalueCylinder} (see Section~\ref{section:general}),
  it follows that $\sigma_1(M_\ell)\to\infty$.
\end{rem}

\begin{rem}
Given a fixed submanifold $\Sigma\subset\R^m$ of dimension $n-1$, we
do not know if there exists a sequence $M_j\subset\R^m$ of
$n$-dimensional submanifolds with boundary $\Sigma$ such that
$\lim_{j\to\infty}\sigma_1(M_j)=\infty$. The above theorems show that
if 
such a sequence exists, then 
the volumes $|M_j|$ must also tend to infinity. Note
that in the Riemannian setting, there exists an example of a compact
manifold $M$ with connected boundary $\Sigma$ and a sequence of
Riemannian metrics $g_j$ such that the $g_j$ coincide on $\Sigma$
while $\sigma_1(M,g_j)\to\infty$ as $j \to \infty$ and the volume of $(M,g_j)$ is
bounded. See \cite{cg}.
\end{rem}

Regarding lower bounds, we first note that there are only a few known results.
A general lower bound was obtained by Jammes \cite{jammes1}, in terms of the Cheeger
constant $h(M)$ of $M$ and of a new
Cheeger-type constant $j(M)$. He proved that
$\sigma_1(M)\geq \frac{h(M)j(M)}{4}$. However, it is easy to construct examples where
$\sigma_1$ is bounded away from zero while $h(M)j(M)$ becomes
small. See also \cite{esco1} for a related lower bound and
\cite{HassaMiclo} for lower bounds on $\sigma_k$, $k \geq
1$. Antoine Métras has recently constructed an example of a
  sequence $M_j\subset\R^m$ with prescribed boundary $\Sigma$ such that
  $\lim_{j\to\infty}\sigma_1(M_j)=0$. This example and its developments will
be published elsewhere.

Nevertheless, in the particular context of hypersurfaces of revolution
(see Section~\ref{warped} for the precise definition), it is not difficult
to obtain lower and upper bounds because a neighborhood of the
boundary is quasi-isometric to a cylinder (see Section~\ref{warped}).
By working directly with the Min-Max characterization of
eigenvalues and using Fourier decomposition, it is possible to obtain
sharp lower bounds.
\begin{thm} \label{thm:lowerboundrev1bdry}
  For each $k \geq 1$, any hypersurface of revolution $M\subset\R^{n+1}$ with boundary
  $\Sp^{n-1}\times\{0\}$ satisfies
  $\sigma_k(M) \ge \sigma_k(\mathbb B^{n})$,
  with equality if and only if $M=\mathbb B^{n}\times\{0\}$.
\end{thm}

  \begin{rem}
  Let $\Sigma$ be an $(n-1)$-dimensional, connected, smooth
  hypersurface in $\R^n$ bounding a domain $\Omega$. It is not true in
  general that
  $\sigma_k(\Omega)\leq\sigma_k(M)$ for each compact manifold $M$ with
  $\Sigma=\partial M$. The domain $\Omega$ which is shown on the left
  side of Figure  \ref{figure:example} has a thin passage in its complement. One can
  consider a surface $M$ containing the passage and with $\partial
  M=\partial\Omega$. This could for instance be obtained by gluing a
  spherical cap to the domain on the right of the figure. If the
  passage is thin enough, any eigenvalue
  $\sigma_k(M)$  will be arbitrarily small. See \cite[Section 4.1]{gpsurvey}.
  \begin{figure}
    \includegraphics[width=9cm]{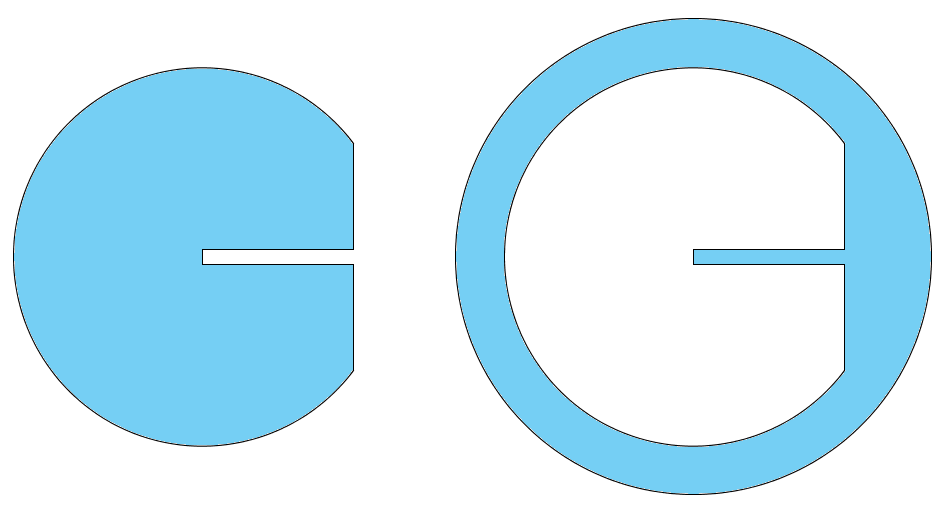}
    \caption{A domain with a thin passage in its complement
    \label{figure:example}}
  \end{figure}
\end{rem}

\begin{rem}
  In the situation where $n=2$, we will prove that each surface of revolution $M\subset\R^3$ with
  boundary $\Sp^1\subset\R^2\times\{0\}$ has the same spectrum as the unit
  disk: $\sigma_k(M)=\sigma_k(\D)$ for each $k\in\N$. It is possible to
  deduce this from the Osgood-Phillips-Sarnak uniformization theorem
  \cite{OPS}, but we give a direct simple proof in Appendix
  A. As far as the authors are aware, this is the first time
    that this isopectrality result appears in the literature.
  In dimension at least $3$, it is not difficult to
  construct examples of hypersurfaces of revolution where $\sigma_k(M)>\sigma_k(\mathbb{B}^n).$ See
  Example \ref{example:mixed}.
\end{rem}

Similar methods can also be applied to hypersurfaces of revolution with two boundary
components that are located in $\R^n \times \{0\}$ and $\R^n \times \{d\}$ respectively.
\begin{thm}\label{thm:lowerboundrev2bdry}
  Let $M\subset\R^{n+1}$ be a hypersurface of revolution with
  boundary $\Sp^{n-1}\times\{0\}\cup
  \Sp^{n-1}\times\{d\}\subset\R^{n+1}$.
  Let $L>0$ be the intrinsic distance between these two components.
  If $L\geq 2$, then for each $k \ge 1$,
  \begin{equation*}
    \sigma_k(M) \ge \sigma_k(\mathbb B^n\sqcup \mathbb B^n),
  \end{equation*}
  with equality if and only if $M=\mathbb B^n\sqcup \mathbb B^n$.
  In particular, for $d\geq 2$ this is always true.
\end{thm}

\begin{rem}
  The lower bounds that we have obtained for
  hypersurfaces of revolution do not require any
  control of the curvature. This also holds for the upper bounds that we obtain in Section~\ref{warped}.
  This is in contrast to the more general
  results given above and should also be
  compared with many of the known results (See \cite{prostu,esco1, Carreno}).
  In addition, our results for hypersurfaces of revolution do not depend on the volume of $M$, as
  opposed to those in \cite{ceg2}.
  Note also that for surfaces, the general bounds given by
    Theorem~\ref{thm:verygeneralupperbound} and Theorem~\ref{thm:generalupperbound}
    do not depend on the genus, as opposed to those presented in \cite{fraschoen,GPsurfaces}.
\end{rem}

\subsection*{Plan of the paper}
In Section~\ref{section:general} we will recall some general facts about the
Steklov problem and properties of its spectrum.
In Section~\ref{section:upperboundrevol} the proofs of
Theorem~\ref{thm:lowerboundrev1bdry},
and Theorem~\ref{thm:lowerboundrev2bdry} will
be presented. They are based on Fourier decomposition
and comparison arguments.
The proofs of Theorem~\ref{thm:verygeneralupperbound} and Theorem~\ref{thm:generalupperbound}
are more involved, and they are based on a method
that was introduced in \cite{cm}. They will
be presented in Section~\ref{section:verygeneralupperbound}. Finally, the
isospectrality of compact surfaces of revolution with connected boundary will be presented in an appendix.

 \section{Some general facts about Steklov and mixed
   problems} \label{section:general}

The Steklov eigenvalues of $(M,g)$ can be characterized by the following variational formula.
\begin{equation}
  \label{minmax}
  \sigma_j(M) = \min_{E \in \mathcal{H}_j} \max_{0 \neq u \in E} R_M(u),
\end{equation}
where $\mathcal{H}_j$ is the set of all $j$-dimensional subspaces in the Sobolev space $H^1(M)$
which are orthogonal to constants on $\Sigma$, and
\begin{equation*}
  \label{Rayleighquotient}
  R_M(u) = \frac{\int_{M} \vert \nabla u \vert^2 dV_M}{\int_{\Sigma} \vert u \vert^2 dV_{\Sigma}}
\end{equation*}
is the Rayleigh quotient.

In order to obtain upper or lower bounds for the spectrum, it is often
convenient to compare the Steklov spectrum with the spectrum of mixed
Steklov-Dirichlet or Steklov-Neumann problems on domains $A\subset M$
such that $\Sigma \subset A$. We denote by $\partial_IA$ the interior
boundary of $A$ (that is the intersection of the boundary of $A$ with
the interior of $M$) and we suppose that it is smooth.

The mixed Steklov-Neumann problem
on $A$ is the eigenvalue problem
\begin{gather*}
  \Delta f=0\mbox{ in } {A},\\
  \partial_\nu f=\sigma f\mbox{ on } \Sigma,\qquad
  \partial_\nu f=0 \mbox{ on }\ \partial_IA,
\end{gather*}
 where $\nu$ denotes the outward normal to $\partial A$.
The eigenvalues of this mixed problem form a discrete sequence
$$0=\sigma_0^N( A) \le \sigma_1^N( A)\leq\sigma_2^N(A)\leq\cdots\nearrow\infty,$$
and for each $j\geq 1$ the $j$-th eigenvalue is given by
\begin{gather*}
 \sigma_j^{N}(A)=\min_{E\in\mathcal{H}_{j}(A)}\max_{0\neq u\in E}\frac{\int_{M}|\nabla u|^2\,dV_M}{\int_{\Sigma}|u|^2\,dV_\Sigma},
\end{gather*}
where $\mathcal{H}_{j}(A)$ is the set of all $j$-dimensional subspaces
in the Sobolev space $H^1(A)$ which are orthogonal to constants on
$\Sigma$.

The mixed Steklov-Dirichlet problem
on $A$ is the eigenvalue problem
\begin{gather*}
  \Delta f=0\mbox{ in } {A},\\
  \partial_\nu f=\sigma f\mbox{ on } \Sigma,\qquad
  f=0 \mbox{ on }\ \partial_IA.
\end{gather*}
The eigenvalues of this mixed problem form a discrete sequence
$$0<\sigma_0^D( A)\leq\sigma_1^D(A)\leq\cdots\nearrow\infty,$$
and the $j$-th eigenvalue is given by
\begin{gather*}
 \sigma_j^{D}(A)=\min_{E\in\mathcal{H}_{j,0}(A)}\max_{0\neq u\in E}\frac{\int_{A}|\nabla u|^2\,dV_M}{\int_{\Sigma}|u|^2\,dV_\Sigma},
\end{gather*}
where $\mathcal{H}_{j,0}(A)$ is the set of all $(j+1)$-dimensional subspaces
in the Sobolev space
$H_0^1(A)=\{u\in H^1(A)\,:\, u=0 \mbox{ on }\partial_IA\}.$

\medskip

Comparisons with the Rayleigh quotient give the following
bracketing for each $j\in\N$:
\begin{gather}\label{ineq:compmixed}
\sigma_j^N(A) \le \sigma_j(M) \le \sigma_{j}^D(A).
\end{gather}
Note in particular that for $j=0$ this simply confirms that
$\sigma_0^N(A)=0\leq\sigma_0(A)<\sigma_0^D(A)$.

\medskip

In the particular case where $A$ is a cylinder, separation of
variables leads to the following which is similar to \cite[Lemma 6.1]{ceg2}.
\begin{lemma}\label{lemma:EigenvalueCylinder}
Let $\Sigma$ be a closed, connected Riemannian manifold, and let
$$0=\lambda_0<\lambda_1\leq\lambda_2\leq\cdots\nearrow\infty$$
be the spectrum of the
Laplace operator $\Delta_\Sigma$ on $\Sigma$. Let $(u_k)$ be an orthonormal basis of $L^2(\Sigma)$ such that
$$\Delta_{\Sigma}u_k=\lambda_ku_k.$$
Consider the cylinder
$C_L=\Sigma\times[0,L]$ of length $L>0$.
  \begin{enumerate}
  \item The Steklov eigenvalues of $C_L$ are
    $$0, 2/L, \sqrt{\lambda_k}\tanh(\sqrt{\lambda_k}L/2),
    \sqrt{\lambda_k}\coth(\sqrt{\lambda_k}L/2), \quad\mbox{ for }k \geq 2.$$
    The corresponding eigenfunctions are $1$, and $r-L/2$, as well as
    $$u_k(p)\cosh\left(\sqrt{\lambda_k}(r-L/2)\right),\quad
    u_k(p)\sinh\left(\sqrt{\lambda_k}(r-L/2)\right),\quad k\geq 2.$$

\item   The Steklov-Neumann (i.e. sloshing) eigenvalues on $C_L$
  are
  $$\sigma_k^N(L):=\sqrt{\lambda_k}\tanh(\sqrt{\lambda_k}L),\quad k\geq 0.$$
  with corresponding eigenfunctions
  $$u_k(p)\frac{\cosh(\sqrt{\lambda_k}(L-r))}{\cosh(\sqrt{\lambda_k}L)}.$$
  In particular $\sigma_0^N(L)=0$, with constant eigenfunction.
  \item The Steklov-Dirichlet eigenvalues on $C_L$ are
  $\sigma_1^D(L)=1/L$ with corresponding eigenfunction
  $1-r/L$ and for each $k\geq 2$,
  $$\sigma_k^D(L)=\sqrt{\lambda_k}\coth(\sqrt{\lambda_k}L).$$
  with corresponding eigenfunctions
  $$u_k(p)\frac{\sinh(\sqrt{\lambda_k}(L-r))}{\sinh(\sqrt{\lambda_k}L)}.$$
  \end{enumerate}
\end{lemma}

\begin{ex}\label{example:mixed}
  Given $L\geq 0$ let $M_L$ be a compact manifold with connected boundary such
  that a neighborhood of the boundary is isometric to $\Sp^{n-1}\times
  [0,L]$. It follows from Inequality \eqref{ineq:compmixed} that
    \begin{gather*}
      \sigma_k^N(\Sp^{n-1}\times[0,L]) \le \sigma_k(M) \le \sigma_k^D(\Sp^{n-1}\times
            [0,L]).
    \end{gather*}
    As above, we have for $k\geq 1$ that
    \begin{gather*}
      \sigma_k^D(\Sp^{n-1}\times[0,L])=\sqrt{\lambda_k}\coth(\sqrt{\lambda_k}L)\\
      \sigma_k^N(\Sp^{n-1}\times[0,L])=\sqrt{\lambda_k}\tanh(\sqrt{\lambda_k}L).
    \end{gather*}
    In particular, when $L\to\infty$,
    $$\sigma_k(M_L)=\sqrt{\lambda_k}+o(1).$$
    For $n=2$, one has $\sigma_k(M_L)=\sqrt{\lambda_k}$ since all surfaces of
    revolution with connected boundary are Steklov isospectral (see Appendix \ref{section:isospectral}). For $n\geq 3$ this proves that
    $\sigma_k(M_L)\neq \sigma_k(\mathbb{B}^n).$ Indeed the Steklov
    spectrum of $\mathbb{B}^n$ is the sequence  of all positive integers (see \cite[Example 1.3.2]{gpsurvey}), but
    the eigenvalues of the Laplacian on the sphere $\Sp^{n-1}$ are not
    perfect squares (see \cite[Chapter III. C]{bgm}).
\end{ex}

In Section~\ref{warped} we will also need to compare the eigenvalues of
quasi-isometric cylinders.
\begin{obs}\label{obs:quasiisocontrol}
 Let $M$ be a compact manifold of dimension $n$, with smooth boundary
 $\Sigma$ and let $g_1,g_2$ be two Riemannian metrics on $M$ which coincide on the boundary $\Sigma$ and which are
 quasi-isometric with ratio $A\geq 1$, which means that for each $x\in
 M$ and $0\neq v \in T_xM$ we have
 $$
 \frac{1}{A} \le \frac{g_1(x)(v,v)}{g_2(x)(v,v)} \le A.
 $$
 Then the Steklov eigenvalues with respect to $g_1$ and $g_2$ satisfy the following inequality:
 $$
 \frac{1}{A^{\frac{n}2+1}}\le \frac{\sigma_k(M,g_1)}{\sigma_k(M,g_2)} \le A^{\frac{n}2+1}.
 $$
 The same conclusion holds for the mixed eigenvalues $\sigma_k^N$ and
 $\sigma_k^D$.
\end{obs}
This follows directly from the Min-Max characterization \eqref{minmax}
(see \cite[Proposition 32]{discretizationBCG} and \cite{Dodziuk} for
related statements).

\section{Bounds for hypersurfaces of revolution}
\label{section:upperboundrevol}

\subsection{Hypersurfaces of revolution}\label{warped}
By \emph{hypersurface of revolution} we mean there
is a parametrisation $\psi:\Sp^{n-1}\times [0,L]\rightarrow\R^{n+1}$ of $M$
given by
$$\psi(p,r)=h(r)p+z(r)e_{n+1},$$
where $\Sp^{n-1} \subset \R^{n}$ is the unit sphere, and the smooth functions
$h,z\in C^\infty([0,L])$ satisfy $h(0)=1$ and $h(r)>0$ for $0<r<L$.
The curve $$r\mapsto h(r)p+z(r)e_{n+1}$$ is assumed to be parametrized
by its arc-length, so that $h'^2+z'^2\equiv 1$. It
follows that  $|h'|\leq 1$ on $[0,L]$, and
integration leads to the following inequalities. For each $r\in (0,L)$,
\begin{gather}
  1-r\leq h(r)\leq 1+r,\label{constraints_revol_typeA}\\
    h(L)-r\leq h(L-r)\leq h(L)+r.\label{constraints_revol_typeB}
\end{gather}
In the coordinates given by the map $\psi$, the induced Riemannian metric $g$ on $[0,L]\times
\Sp^{n-1}$ has the form
\begin{gather*}
  g=dr^2+h(r)^2g_{0},
\end{gather*}
where $g_{0}$ is
the usual canonical metric on $\Sp^{n-1}$.
In particular, the situation where $M=\mathbb B^{n}\times\{0\}$ corresponds to
$h(r)=1-r$.
We note that the restriction of $g$ 
to $\Sp^{n-1}\times[0,L/2]$ is quasi-isometric to the product metric
$dr^2+g_{0}$ with ratio $(1-L/2)^{-2}$.

\begin{prop}\label{prop:upperboundrev1bdry}
  Let $M\subset\R^{n+1}$ be a hypersurface of revolution with
  connected boundary. Then,
  \begin{gather}\label{ineq:revolc}
    \frac{1}{4^{\frac{n}2+1}}\sigma_k^N(1/2)\leq\sigma_k(M,g) \le
  4^{\frac{n}2+1}\sigma_k^D(1/2),
  \end{gather}
  where $\sigma_k^N(1/2),\sigma_k^D(1/2)$ are the eigenvalues of mixed
  Steklov-Neumann and Steklov-Dirichlet problems on the cylinder
  $\Sp^{n-1}\times[0,1/2]$.
  In particular, there are constants $b(n)\geq a(n)>0$ such that
  \begin{gather}\label{ineq:compweyl}
      a(n) k^{1/(n-1)}\leq \sigma_k(M,g) \le b(n)
  k^{1/(n-1)}\quad\forall k\geq 1.
  \end{gather}
\end{prop}
\begin{proof}[Proof of \ref{prop:upperboundrev1bdry}]
  It follows from \eqref{constraints_revol_typeA} that on
  $[0,1/2]\times\Sp^{n-1}$ the metric $g=dr^2+h(r)^2g_0$ is
  quasi-isometric to $dr^2+g_0$ with ratio $4$. Together with the
  comparison inequality \eqref{ineq:compmixed} and Observation~\ref{obs:quasiisocontrol}, this leads to \eqref{ineq:revolc}.
  Inequality \eqref{ineq:compweyl} follows from Weyl's law for the eigenvalues
  $\lambda_k$ of the Laplacian on $\Sp^{n-1}$ and Lemma \ref{lemma:EigenvalueCylinder}, since $\tanh(r)=1+o(r^{-\infty})$  and $\coth(r)=1+o(r^{-\infty})$ as $r\to\infty$.
\end{proof}

\begin{rem} The above construction shows that our bound has no reason to be
  sharp. In particular, we could use a cylinder of different length to
  match the problem better. It is an \emph{open problem} to find the
  maximum of $\sigma_k$ among hypersurfaces of revolution.
\end{rem}

We will also obtain an upper bound for hypersurfaces of revolution with
two boundary components, which will be proved using Fourier
decomposition of a
function $f\in C^\infty(\Sp^{n-1}\times [0,L])$. See Proposition~\ref{prop:upperboundrev2bdry}
below. Note that these computations will also be
useful in the proof of the sharp lower bounds from Theorem~\ref{thm:lowerboundrev1bdry}
and Theorem~\ref{thm:lowerboundrev2bdry}.

Let $0=\lambda_0<\lambda_1\leq\lambda_2\leq\cdots\nearrow+\infty$ be the
spectrum of the Laplacian on $\Sp^{n-1}$ and consider $(S_j)_{j\in\N}$
a corresponding orthonormal basis of $L^2(\Sp^{n-1})$.
Any smooth function $u$ on $M$ admits a Fourier series expansion
\begin{equation*}
  u(r,p) = \sum_{j=0}^{\infty} a_j(r) S_j(p).
\end{equation*}

We also normalize the function $u$ by $\int_{\Sigma} u(r,p)^2 \, dV_\Sigma = 1$ so that
$\sum_{j=1}^{\infty} a_j(0)^2 = 1$.
We have that
$$du(r,p) = \sum_{j=0}^{\infty} (a_j'(r)S_j(p) dr + a_j(r) dS_j(p)),$$
hence
\begin{align}\label{formulaenergy}
  \Vert du(r,p) \Vert_{g}^2
  &=
  \sum_{j=0}^{\infty} \int_0^L \Big(a_j'(r)^2 + \frac{a_j(r)^2 \lambda_j(\mathbb{S}^{n-1})}{h(r)^2}\Big) h(r)^{n-1} \, dr\\
  &=
  \sum_{j=0}^{\infty} \int_0^L (a_j'(r)^2 h(r)^{n-1} + a_j(r)^2 h(r)^{n-3}\lambda_j(\mathbb{S}^{n-1})) \, dr.\nonumber
\end{align}

\begin{prop}\label{prop:upperboundrev2bdry}
  Let $M\subset\R^{n+1}$ be a hypersurface of revolution with dimension
  $n\geq 3$ and
  boundary $\Sp^{n-1}\times\{0\}\cup
  \Sp^{n-1}\times\{d\}\subset\R^{n+1}$.
  Let $L>0$ be the intrinsic distance between these two
  components. For each $k\in\N$,
  \begin{equation*}
    \sigma_k(M) \leq (1+L)^{n-1}\sigma_k([0,L]\times\Sp^{n-1}),
  \end{equation*}
  where the right-hand side tends to 0 when $L\to 0$.
\end{prop}

\begin{proof}[Proof of Proposition~\ref{prop:upperboundrev2bdry}]
  Let $u$ be a solution of the Steklov problem on the cylinder
$C_L=[0,L]\times\Sp^{n-1}$, normalized by $\int_{\partial C_L}u^2=1.$
The Dirichlet energy of $u$ for the metric $g=dr^2+h(r)^2g_{0}$ is
given by formula \eqref{formulaenergy}, and it follows from
$h(r)<1+r\leq 1+L$ and $n\geq 3$
that
\begin{align*}
  \Vert du(r,p) \Vert^2
  &\leq
  \sum_{j=0}^{\infty} \int_0^L (a_j'(r)^2 (1+L)^{n-1} + a_j(r)^2
  (1+L)^{n-3}\lambda_j(\mathbb{S}^{n-1})) \, dr\nonumber\\
  &\leq (1+L)^{n-1}
  \sum_{j=0}^{\infty} \int_0^L (a_j'(r)^2 +
  a_j(r)^2\lambda_j(\mathbb{S}^{n-1})) \, dr\\
  &=(1+L)^{n-1}\sigma_k(C_L).
\end{align*}
Using the Min-Max characterization \eqref{minmax} completes the  proof.
\end{proof}

\subsection{Lower bounds for revolution hypersurfaces with connected boundary}

The goal of this section is to prove Theorem~\ref{thm:lowerboundrev1bdry}. The idea is to transplant a function of
$M$ to a function on the unit ball and to take it as a test
function for the Steklov operator. However, this function is in general not continuous at the
origin, but it is in the Sobolev space $H^1(\mathbb{B}^n)$ because the
capacity of a point is $0$.
\begin{lemma}\label{lemma:sobolev}
If $f\in C^\infty(\mathbb{B}^n\setminus\{0\})$ is bounded and satisfies $\Vert\nabla f\Vert_{L^2(\mathbb{B}^n)}<\infty$, then $f\in H^1(\mathbb{B}^n)$.
\end{lemma}
\begin{proof}[Proof of Theorem~\ref{thm:lowerboundrev1bdry}]
We use the fact that $n\ge 3$, because, as mentioned above, all
hypersurfaces of revolution with connected boundary
are isospectral to the disk.

Recall that $h$ is defined on
$M=[0,L]\times \Sp^{n-1}$ and we have $h(0)=1$, $\vert h'(r)\vert \le
1$,  $h(L)=0$ and $L\ge 1$ since the revolution manifold $M$ has
connected boundary. Recall from \eqref{constraints_revol_typeA} that
this implies that on $[0,1]$,
we have $h(r)\ge 1-r$. Substituting this into \eqref{formulaenergy}, using
$\sum_{j=1}^{\infty} a_j(0)^2 = 1$, and the fact that $n\ge 3$
leads to
\begin{align}\label{relation}
  \sigma_k(M)&=\Vert du(r,p) \Vert_{g}^2\\
  &\geq
  \sum_{j=0}^\infty\int_0^1 (a_j'(r)^2 (1-r)^{n-1} + a_j(r)^2 (1-r)^{n-3}\lambda_j(\mathbb{S}^{n-1})) \, dr.\nonumber
\end{align}
Consider the first $k$ normalized eigenfunctions $u_1,...,u_k$ of $M$
corresponding to the first $k$ positive Steklov eigenvalues
$\sigma_1(M), \dots, \sigma_k(M)$. To each of them associate a
function on the unit ball which is orthogonal to the constant functions on
$\mathbb{S}^{n-1}=\partial M$ simply by taking the restriction to the
ball ($r \in [0,1]$).
Despite not being continuous at $r=1$, it follows from Lemma \ref{lemma:sobolev} that
these functions can be considered as test
functions for the Steklov operator on the ball.
It follows from the Min-Max principle \eqref{minmax} that
$$
\sigma_k(\mathbb B^n)= \min_{E\in\mathcal{H}_k} \max_{\{u \in E, u\not=0\}} R_{\mathbb B^n}(u).
$$
As a specific
vector space, we choose the space $E_0$ given by the restriction of
the first $k$ eigenfunctions $u_1,...,u_k$ of $M$ to the ball $\mathbb
B^n$. Then we have
$$
\sigma_k(\mathbb B^n) \le \max_{\{u \in E_0, u\not=0\}} R_{\mathbb B^n}(u),
$$
and by (\ref{relation}), for each $u \in \{u_1,...,u_k\}$,
$$
R_M(u) \ge R_{\mathbb B^n}(u\bigl|\bigr._{\mathbb B^n}).
$$
The proof is completed by observing that
$$
\sigma_k(M)= \max_{\{u \in \{u_1,...,u_k\}\}}R_M(u) \ge \max_{\{u \in E_0, u\not=0\}} R_{\mathbb B^n}(u) \ge \sigma_k(\mathbb B^n).
$$
In the case of equality, we must have $L=1$ and $h(r)=1-r$ which corresponds to the unit ball $\mathbb B^n$.
\end{proof}

\subsection{Lower bounds for hypersurface of revolution with two
  boundary components}

The goal of this section is to prove Theorem~\ref{thm:lowerboundrev2bdry}.
The proof is similar to the proof of
Theorem~\ref{thm:lowerboundrev1bdry}.

\begin{proof}[Proof of Theorem~\ref{thm:lowerboundrev2bdry}]
  We proceed by comparison with the metric $dr^2+(1-r)^2g_0$ for $r
  \in [0,1]$, as well as  with the metric
  $dr^2+ (r-L+1)^2$ for $r \in [L-1,L]$.

  Let $\lambda_j=\lambda_j(\Sp^{n-1})$ be the eigenvalues of Laplace
  operator on the sphere $\Sp^{n-1}$.
  For $k \in \N$, let $\{a_j(r)\}_{j \in \N}$ be such that $\sum_{j=1}^{\infty} a_j(0)^2 = 1$ and
\begin{equation*}
  \sigma_k(M) =
  \sum_{j=1}^{\infty} \int_0^L (a_j'(r)^2 h(r)^{n-1} + a_j(r)^2 h(r)^{n-3}\lambda_j) \, dr.
\end{equation*}
By \eqref{constraints_revol_typeA} and \eqref{constraints_revol_typeB}, and since $L \ge 2$,
\begin{align}\label{e1.5a}
  \sigma_k(M) & \geq \sum_{j=1}^{\infty} \int_0^1 (a_j'(r)^2 (1-r)^{n-1} + a_j(r)^2 (1-r)^{n-3}\lambda_j) \, dr \\
  & + \sum_{j=1}^{\infty} \int_{L-1}^L (a_j'(r)^2 (r-L+1)^{n-1} + a_j(r)^2 (r-L+1)^{n-3}\lambda_j) \, dr.\notag
\end{align}
Let $u_1, \dots, u_k$
be the normalized Steklov eigenfunctions corresponding to $\sigma_1(M),
\dots, \sigma_k(M)$. These give rise to $k$ test functions on $\mathbb{B}^{n} \sqcup \mathbb{B}^{n}$ by restricting
$r$ to $[0,1] \sqcup [L-1,L]$. These test functions are not continuous
at the centers of these balls, but Lemma \ref{lemma:sobolev} shows that this is not a problem.
Call this collection of functions $E_0$. Then,
by \eqref{e1.5a} and the Min-Max characterization \eqref{minmax}, it
follows that
\begin{align*}
\sigma_k(\mathbb{B}^{n} \sqcup \mathbb{B}^{n})
&\leq \max_{u \in E_0} R_{\mathbb{B}^{n} \sqcup \mathbb{B}^{n}}(u) \notag\\
&\leq \max_{u \in E_0} R_{M}(u) = \sigma_k(M).
\end{align*}
In the case of equality, we must have $L=2$ and $h(r)=1-r$ which corresponds to
the case of $\mathbb{B}^n \sqcup \mathbb{B}^n$.
\end{proof}

\section{Upper bounds for general submanifolds}
\label{section:verygeneralupperbound}

\subsection{Some metric geometry}
The method we will use to obtain upper bounds for Steklov
eigenvalues is to construct disjointly supported test functions for the Rayleigh
quotient. These functions are obtained by constructing disjoint
domains that are both ``heavy enough'' and
far enough away from each other. The test function
associated to one domain takes the value $1$ on it, and $0$ outside a convenient neighbourhood. Our main
task is therefore to construct such domains. This is the goal of the
next Lemma.
\begin{lemma}\label{CMrevisited}
  Let $(X,d,\mu)$ be a complete, locally compact metric measured
  space, where $\mu$ is a non-atomic finite measure. Assume that for all
  $r>0$, there exists an integer $C$ such that each ball of radius
  $r$ can be covered by $C$ balls of radius $r/2$.
  Let $K>0$. If there
  exists a radius $r>0$ such that,
  for each $x \in X$
  $$
  \mu(B(x,r)) \le \frac{\mu(X)}{4C^2K},
  $$
  then, there exist $\mu$-measurable subsets $A_1,...,A_K$ of $X$
  such that, $\forall i\le K$, $\mu(A_i)\ge \frac{\mu(X)}{2CK}$
  and, for
  $i\not =j$, $d(A_i,A_j) \ge 3r$.
\end{lemma}
This result was developed in \cite[Lemma 2.2 and Corollary 2.3]{cm}.
The version given above is inspired by that of \cite[Lemma 2.1]{ceg1}, with
a further simplification. Indeed,
the original statement involves a number $N(r)\in\N$ which depends on
the radius rather than a packing constant $C$.
In our present
setting, because the ambient space is $X=\R^n$ with its usual
Euclidean distance, it is clear that the packing constant $C$ does not
depend on the radius. In fact, $C=C(n):=32^n$ is a good choice.
Note also that these ideas are
explained with more details in \cite{ColboisMTL}.

The measure $\mu$ will be the measure
associated to the fixed submanifold $\Sigma$. That is, for a Borelian
subset $\mathcal{O}$ of $\R^{m}$, we take (as in \cite{ceg1})
$\mu(\mathcal O)=\int_{\Sigma \cap \mathcal O} dV_\Sigma$.
In particular, $\mu (\Sigma)$
is the usual volume $\vert \Sigma \vert$ of $\Sigma$. The number $K$
is related to the number of eigenvalues we want to estimate. In order
to estimate $\sigma_k$, it turns out that we need to begin with the
construction of $(2k+2)$ test functions, so we will take $K=2k+2$.

The proof of Theorem~\ref{thm:verygeneralupperbound} will be an
application of Lemma \ref{CMrevisited}. The proof of Theorem
\ref{thm:generalupperbound} is more tricky: we will first construct a
family of disjoint balls, following the approach used
in \cite{ceg1}, and
only in a further step, we will use Lemma \ref{CMrevisited} in order
to solve a particular case.

\subsection{Proof of the upper bound for general submanifolds}
\label{section:proofverygeneral}
The goal of this section is to prove Theorem~\ref{thm:verygeneralupperbound}.
The constant $A$ which appears in its
statement will depend on the distortion of a connected submanifold $N\subset\R^m$:
\begin{gather}\label{def:distortion}
  \mbox{disto}(N):=\sup_{x\neq y\in N}\frac{d_N(x,y)}{\|x-y\|},
\end{gather}
where $d_N$ is the usual geodesic distance on $N$ and $\|\cdot\|$ is the usual Euclidean norm.

\begin{proof}
  Let $\Sigma_1,\cdots,\Sigma_b$ be the connected components of $\Sigma$.
  Because $\Sigma$ is compact, the number
  $$\gamma:=\max_{1\leq i\leq b}\mbox{disto}(\Sigma_i)$$
  is finite. It follows that for any
  $r>0$, any $x\in\Sigma$ and each $i\in\{1,\cdots,b\}$,
  $$B_{\R^m}(x,r)\cap\Sigma\subset \cup_{i=1}^bB_{\Sigma_i}(x,\gamma
  r).$$
  Here and further $B_{\R^m}(x,r)$ is the usual Euclidean ball of
  radius $r$ centered at $x$, and $B_{\Sigma_i}(x,\gamma r)$ is an
  intrinsic geodesic ball in $\Sigma_i$.
  Note that for $x\in \R^m$ the ball $B_{\R^m}(x,r)$ is of
  $\mu$-measure $0$ if and only if it does not intersect any of the boundary components $\Sigma_i$. Therefore, in the
  situation where it intersects $\Sigma_i$, there exists a point $y_i\in \Sigma_i$ such that
  $$B_{\R^m}(x,r)\cap\Sigma_i\subset B_{\R^m}(y_i,2r)\cap\Sigma_i\subset B_{\Sigma_i}(y_i,2\gamma r).$$
  The compactness of $\Sigma$ also implies a lower bound on the Ricci curvature.
  Hence it follows from the Bishop-Gromov theorem that there exists $\eta>0$
  such that any ball $B_\Sigma(r)$ of radius $r$ in $\Sigma$ satisfies
  $$\vert B_\Sigma(r)\vert \le \eta r^{n-1}.$$
  In fact, $\eta$ may be bounded above  in terms of the lower bound on
  the Ricci curvature and an upper bound for the diameter of $\Sigma$.
  Together with the above, this implies that
  $$\mu(B_{\R^m}(x,r))\leq b\eta (2\gamma r)^{n-1}=\delta r^{n-1},$$
  where $\delta=b\eta(2\gamma)^{n-1}$.


  We apply Lemma~\ref{CMrevisited} with $K=2k+2$ and
  $$r=r_k=\left(\frac{\vert \Sigma\vert}{8C(m)^2(k+1)\delta}\right)^{1/(n-1)},$$
  so that the inequality
  $$\mu(B(x,r)) \le \frac{\mu(\Sigma)}{4C(m)^2(2k+2)}$$
  is satisfied.

\smallskip
We get $2k+2$ subsets $A_1,...,A_{2k+2}$ in $\R^{m}$ such that

$$\mu(A_i)\ge \frac{\vert\Sigma \vert}{4C(m)(k+1)}\quad\mbox{ for each }
i$$
and, if
$i\not=j$, $d(A_i,A_j) \ge 3r_k$.

\smallskip
For each $i \in \{1,\dots,2k+2\}$, define the $r_k$-neighborhood of $A_i$ as
\begin{equation*}
A_i^{r_k} = \{x \in \Omega : d(x,A_i)<r_k\} \subset \R^{n+1}.
\end{equation*}

\smallskip
Since $A_1^{r_k}, \dots, A_{2k+2}^{r_k}$ are disjoint, there exist $k+1$ of them, say
$A_1^{r_k}, \dots, A_{k+1}^{r_k}$, such that, for $i \in \{1,\dots,k+1\}$, we have
\begin{equation*}
\vert M \cap A_i^{r_k} \vert \leq \frac{\vert M \vert}{k+1}.
\end{equation*}
Similarly to \cite{cm}, for each $i \in \{1,\dots,k+1\}$, we construct a test function $g_i$ with support in
$A_i^{r_k}$ as follows. For each $x \in A_i^{r_k}$,
\begin{equation*}
g_i(x) : = 1 - \frac{d(A_i,x)}{r_k}.
\end{equation*}
Then $\vert \nabla g_i \vert^2 \leq \frac{1}{r_k^2}$ almost everywhere in $A_i^{r_k}$. So we have that,
for each $i \in \{1,\dots,k+1\}$,
\begin{equation*}
\int_{M} \vert \nabla g_i \vert^2 \, dV_M = \int_{M \cap A_i^{r_k}} \vert \nabla g_i \vert^2 \, dV_M
\leq \frac{\vert M \cap A_i^{r_k} \vert}{r_k^2} \leq \frac{\vert M \vert}{(k+1) r_k^2}.
\end{equation*}

The Rayleigh quotient $R(g_i) = \frac{\int_{M} \vert \nabla g_i \vert^2 \, dV_M}{\int_{\Sigma} g_i^2\,dV_\Sigma} $ of $g_i$ becomes

\begin{align*}
  R(g_i)  &\le \frac{\vert M \vert}{(k+1) r_k^2}
  \frac{1}{\mu(A_k)}\\
  &=
  \frac{\vert M \vert}{(k+1) } \frac{4C(m)(k+1)}{\vert \Sigma\vert}\left(\frac{8C(m)^2(k+1)\delta}{\vert \Sigma\vert}\right)^{2/(n-1)},
\end{align*}
which leads to
$$
\sigma_k(M)\le A(n,m,\delta)\frac{\vert M\vert}{\vert \Sigma \vert} \left(\frac{k}{\vert \Sigma\vert}\right)^{2/(n-1)}.
$$
\end{proof}

\subsection{Proof of the upper bound for general submanifolds contained as a
  hypersurface in a linear subspace}
\label{section:generalupperbound}

The goal of this section is to prove Theorem~\ref{thm:generalupperbound}

The idea is to apply the method that was developed in \cite{ceg1} and we recall
it for the convenience of the reader, following closely what was done
in \cite{ceg1} . We cover $\Sigma$ with a family of disjointly
supported sets in $\mathbb R^n$. As $M \subset \R^{m}$, we will
replace each of these sets $A \subset \R^n$ by $A\times \R^{m-n}$ in
order to have disjointly supported domains in $\R^{m}$. We then
construct test functions on these sets to obtain upper bounds for the
Steklov eigenvalues of $M$.

In the first step of the proof of
Theorem~\ref{thm:generalupperbound}, the volume $\vert \Omega\vert$ of
$\Omega$ (the subset of $\R^n$ determined by $\Sigma$) enters into the
game as it did in \cite{ceg1}. The proof is divided into four
steps. In the first two steps, we only work with $\Sigma \subset
\R^n$ and its relationship to $\Omega$. These first two steps are very
close to the method in \cite{ceg1}.
In the third and fourth steps,
we also have to take the volume $\vert M\vert$ of $M$ into
account. This explains the presence of the ratio $\frac{\vert M
  \vert}{\vert \Omega \vert^{(n+1)/n}}$ in the statement of
Theorem~\ref{thm:generalupperbound}.
\begin{proof}[Proof of Theorem~\ref{thm:generalupperbound}]
  \begin{itemize}
  \item[]
  \end{itemize}

\medskip
\noindent
\textbf{First step.}  Fix $k \in \N$. Our first goal is to show that $\Sigma$ cannot be covered by $2(k+1)$ balls
in $\R^n$ each of radius $4r_k$, where
\begin{equation}\label{e1.2a}
  r_k = \biggl(\frac{n \omega_n^{1/n}}{4^{n+1}\rho_{n-1}(k+1)}\biggr)^{1/(n-1)}\vert \Omega \vert^{1/n}
\end{equation}
and where $\omega_n$ denotes the volume of the Euclidean unit ball
$\mathbb B^n$, and $\rho_{n-1}$ the volume of the unit sphere $\Sp^{n-1}\subset \R^n$.

\noindent
Let $x_1, x_2, \dots , x_{2k+2}$ be arbitrary points in $\Omega$, and define
\begin{equation*}
\Omega_0 = \Omega \setminus \cup_{j=1}^{2k+2} B(x_j,4r_k),
\end{equation*}
and
\begin{equation*}
\Sigma_0 = \Sigma \setminus \cup_{j=1}^{2k+2} B(x_j,4r_k).
\end{equation*}
Since $B(x_j,4r_k) \subset \R^n$ for $j \in \{1,2,\dots, 2k+2\}$, we have $\vert B(x_j,4r_k) \vert
< 2 \omega_n(4r_k)^n$ and $\vert \partial B(x_j,4r_k) \vert < 2 \rho_{n-1}(4r_k)^{n-1}$. So
\begin{equation}\label{e1.2d}
\sum_{j=1}^{2k+2} \vert B(x_j, 4r_k) \vert < 4(k+1) \omega_n (4r_k)^n,
\end{equation}
and
\begin{align}\label{e1.2e}
(4r_k)^n &= 4^n\biggl(\frac{n \omega_n^{1/n}}{4^{n+1}\rho_{n-1}(k+1)}\biggr)^{n/(n-1)} \vert \Omega \vert\nonumber\\
&< \frac{\vert \Omega \vert}{16(k+1)}\biggl(\frac{n \omega_n^{1/n}}{\rho_{n-1}}\biggr)^{n/(n-1)}
= \frac{\vert \Omega \vert}{16 \omega_n (k+1)}.
\end{align}
By \eqref{e1.2d} and \eqref{e1.2e}, we have that
\begin{equation*}
\sum_{j=1}^{2k+2} \vert B(x_j, 4r_k) \vert < \frac{\vert \Omega \vert}{4},
\end{equation*}
which implies that
\begin{equation}\label{e1.2g}
\vert \Omega_0 \vert \geq \vert \Omega \vert - \sum_{j=1}^{2k+2} \vert B(x_j,4r_k) \vert
> \frac34 \vert \Omega \vert
\end{equation}
i.e. the union of the balls does not cover $\Omega$.

Now $\partial \Omega_0$ is the union of $\Sigma_0$ and parts of the boundaries of the balls
$\{B(x_j,4r_k) : j=1,2,\dots,2k+2\}$. Hence
\begin{align}\label{e1.2h}
\vert \partial \Omega_0 \vert &\leq \vert \Sigma_0 \vert + \sum_{j=1}^{2k+2} \vert \partial B(x_j, 4r_k) \vert\nonumber\\
&< \vert \Sigma_0 \vert + 4(k+1)\rho_{n-1}(4r_k)^{n-1}\nonumber\\
&= \vert \Sigma_0 \vert + \frac{n\omega_n^{1/n}}{4} \vert \Omega\vert^{(n-1)/n}.
\end{align}
  \begin{figure}
    \includegraphics[width=7cm]{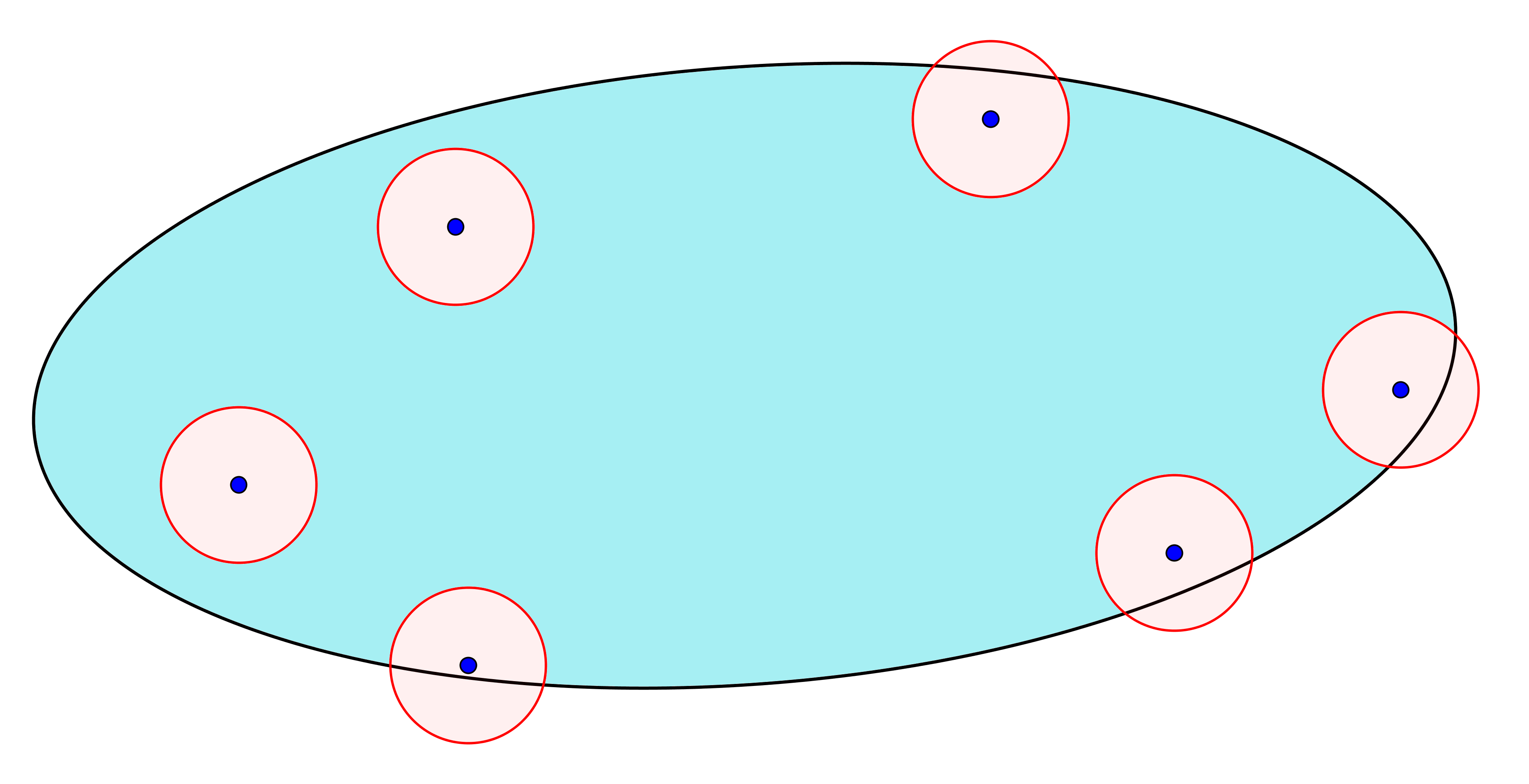}
    \caption{The domain $\Omega_0$ is obtained by removing balls from $\Omega$}
    \label{figure:proof}
  \end{figure}

It follows from the classical Isoperimetric Inequality in $\R^n$ and from
Inequality~\eqref{e1.2g} that
\begin{equation}\label{e1.2j}
\vert \partial \Omega_0 \vert \geq n\omega_n^{1/n} \vert \Omega_0 \vert^{(n-1)/n}
\geq n\omega_n^{1/n} \biggl(\frac34\biggr)^{(n-1)/n} \vert \Omega \vert^{(n-1)/n}.
\end{equation}
So, by \eqref{e1.2h} and \eqref{e1.2j}, we obtain that
\begin{equation}\label{e1.2k}
\vert \Sigma_0 \vert > n\omega_n^{1/n} \biggl(\biggl(\frac34\biggr)^{(n-1)/n} - \frac14\biggr)\vert \Omega \vert^{(n-1)/n} \geq \frac{n\omega_n^{1/n}}{2}\vert \Omega \vert^{(n-1)/n} >0.
\end{equation}
This means that the union of the balls $B(x_j,4r_k)$ does not cover $\Sigma$ and we are done with the first step of the proof.

\medskip
\noindent
\textbf{Second step.} Our goal is to construct balls $B(x_1,r_k),
\dots, B(x_{2k+2},r_k)$ in $\R^n$ with the following
properties:

\begin{enumerate}
\item The balls $B(x_1,2r_k), \dots , B(x_{2k+2},2r_k)$ are mutually disjoint;\\
\item $\mu(B(x_1,r_k)) \geq \mu(B(x_2,r_k)) \geq \dots \geq \mu(B(x_{2k+2},r_k))$;\\
\item For all $x \in \Omega_0 = \Omega \setminus \cup_{j=1}^{2k+2} B(x_j,4r_k)$;
  $$\mu(B(x,r_k)) \leq \mu(B(x_{2k+2},r_k)).$$
\end{enumerate}


Since $\Sigma$ cannot be covered by $2k+2$ balls each of radius $4r_k$, we can construct the collection of $2k+2$
 balls $B(x_1,r_k), \dots, B(x_{2k+2},r_k)$ in $\R^n$ inductively as follows:
\begin{equation*}
\mu(B(x_1,r_k)) = \sup_{x \in \Omega} \mu(B(x,r_k)),
\end{equation*}
and, for $j \in \{2, \dots, 2k+2\}$,
\begin{equation*}
\mu(B(x_j,r_k)) = \sup \{\mu(B(x,r_k)) : x \in \Omega \setminus \cup_{i=1}^{j-1} B(x_i,4r_k)\}.
\end{equation*}

By construction, the balls
$B(x_1,r_k), B(x_2,r_k), \dots, B(x_{2k+2},r_k)$ have non-empty
intersection with $\Sigma$, and
 $\mu(B(x_i,r_k))$ is a monotone decreasing function of $i$.

\medskip
In the sequel, we will consider the two situations
\begin{equation}\label{ineq:case1}
\mu(B(x_{2k+2},r_k)) \geq \frac{n\omega_n^{1/n}}{16 C(n)^2 (k+1)} \vert \Omega \vert^{(n-1)/n}
\end{equation}
and
\begin{equation}\label{ineq:case2}
\mu(B(x_{2k+2},r_k)) \leq \frac{n\omega_n^{1/n}}{16 C(n)^2 (k+1)} \vert \Omega \vert^{(n-1)/n}
\end{equation}
where we recall that $C(n)$ is the packing constant $N(r)$ of $\R^n$
which appears in Lemma~\ref{CMrevisited}.

The theorem will be proved independently in each case. In the
situation where Inequality \eqref{ineq:case1} holds, each of the balls
$B(x_i,r_k)$ will be heavy enough to allow the construction of test functions
supported close to these balls. This is done in the third step
below. In the situation where Inequality \eqref{ineq:case2} holds, it
will be necessary to use Lemma~\ref{CMrevisited} to obtain $k+1$
disjoint sets and proceed with the construction of appropriate test
functions. This will be the fourth and final step of the proof.

\medskip
\noindent
\textbf{Third step.} We suppose that Inequality \eqref{ineq:case1} holds.

For each $j \in \{1,2, \dots, 2k+2\}$,
define the \emph{cylinder}
\begin{equation*}
\bar{B}(x_j,2r_k) = B(x_j,2r_k) \times \R^{m-n} \subset \R^{m}.
\end{equation*}
Since the underlying balls are disjoint, the \emph{cylinders} $\bar{B}(x_1,r_k), \dots, \bar{B}(x_{2k+2},r_k)$ are
disjoint in $\R^{m}$. Then there exist $k+1$ of them, say $\bar{B}(x_1,r_k), \dots, \bar{B}(x_{k+1},r_k)$, such that
for $j \in \{1, \dots, k+1\}$, we have
\begin{equation*}
\vert M \cap \bar{B}(x_j,2r_k) \vert \leq \frac{\vert M \vert}{k+1}.
\end{equation*}

For each $j \in \{1,\dots,k+1\}$ and each $x \in \bar{B}(x_j,2r_k)$, define a test function $f_j$ supported on
$\bar{B}(x_j,2r_k)$ as follows:
\begin{equation*}
f_j(x) : = \min\bigg\{1, 2-\frac{1}{r_k}d(x_j \times \R^{m-n}, x)\bigg\}.
\end{equation*}
Since $\vert \nabla f_j \vert^2 \leq \frac{1}{r_k^2}$ almost everywhere
in $\bar{B}(x_j,2r_k)$, we have that
\begin{align}\label{e1.2t}
\int_{M} \vert \nabla f_j \vert^2 \, dV_M &= \int_{M \cap \bar{B}(x_j,2r_k)} \vert \nabla f_j \vert^2 \, dV_M\nonumber\\
&\leq \frac{\vert M \cap \bar{B}(x_j,2r_k)\vert}{r_k^2} \leq \frac{\vert M \vert}{(k+1) r_k^2}.
\end{align}
We also have that
\begin{align}\label{e1.2u}
\int_{\Sigma} \vert f_j \vert^2 \, dV_\Sigma &= \int_{\Sigma \cap \bar{B}(x_j,2r_k)} \vert f_j \vert^2 \, dV_\Sigma\nonumber\\
&= \int_{\Sigma \cap B(x_j,2r_k)} \vert f_j \vert^2 \, dV_\Sigma \notag\\
&\geq \int_{\Sigma \cap B(x_j,r_k)} dV_\Sigma\nonumber\\
&=\mu(B(x_j,r_k)) \geq \frac{n\omega_n^{1/n}}{16 C(n)^2 (k+1)} \vert \Omega \vert^{(n-1)/n}.
\end{align}
Hence, by \eqref{e1.2t} and \eqref{e1.2u}, using the explicit value of
$r_k$ given in \eqref{e1.2a} we have that the Rayleigh quotient of $f_j$ is
\begin{align*}
R(f_j) &= \frac{\int_{M} \vert \nabla f_j \vert^2 \, dV_M}{\int_{\Sigma} \vert f_j \vert^2 \, dV_\Sigma}
\leq \frac{\vert M \vert}{r_k^2} \frac{16C(n)^2  }{n\omega_n^{1/n} \vert \Omega \vert^{(n-1)/n}} \notag\\
&\leq C_1(n)
\frac{\vert M \vert }{ \vert \Omega \vert^{(n+1)/n}} (k+1)^{2/(n-1)},
\end{align*}
for some constant $C_1(n)$. It follows that
\begin{equation}\label{ineq:finalstep3}
\sigma_k(\Sigma) \leq C_1(n)\frac{\vert M \vert}{\vert \Omega \vert^{(n+1)/n}} (k+1)^{2/(n-1)}.
\end{equation}

\noindent
\textbf{Fourth step.} Now suppose that Inequality \eqref{ineq:case2}
holds.

By construction of the balls, for all $x \in \Omega_0 = \Omega \setminus \cup_{j=1}^{2k+2} B(x_j,4r_k)$,
we have by \eqref{e1.2k} that
\begin{align*}
  \mu(B(x,r_k)) &\leq \mu(B(x_{2k+2},r_k))\nonumber\\
  &\leq \frac{n\omega_n^{1/n}}{16 C(n)^2 (k+1)} |\Omega|^{(n-1)/n}
  \leq\frac{\vert \Sigma_0 \vert}{8 C(n)^2 (k+1)}.
\end{align*}
This  implies that
\begin{equation}\label{e1.2z}
4 C(n)^2 \mu(B(x,r_k)) \leq \frac{\mu(\Sigma_0)}{2k+2}.
\end{equation}
With this in mind, we would like to make use of Lemma~\ref{CMrevisited}.
Consider the metric measured space $(\R^n,d,\mu_0)$ where $d$ is
the Euclidean distance and the measure
$\mu_0$ is the restriction of the measure $\mu$ to $\Sigma_0$. That is,
for any Borel set $U \subset \R^n$,
$\mu_0(U) = \mu(U \cap \Sigma_0)$ so $\mu_0(\R^n) = \mu(\Sigma_0) = \vert \Sigma_0 \vert < +\infty$.

Choose $K=2k+2$ and $r=r_k$. Then \eqref{e1.2z} becomes
\begin{equation*}
4 C(n)^2 \mu(B(x,r_k)) \leq \frac{\mu_0(\Sigma_0)}{2k+2}.
\end{equation*}
Hence, by Lemma~\ref{CMrevisited}, there exist $2k+2$ measurable sets $A_1, \dots, A_{2k+2} \subset \R^n$
such that
\begin{enumerate}
\item $\mu(A_i) \geq \frac{\vert \Sigma_0 \vert}{4 C(n) (k+1)}$
for each $i \in \{1,\dots,2k+2\}$.\\
\item $d(A_i,A_j) \geq 3r_k$ for $i \neq j$.
\end{enumerate}
We then proceed exactly as in the proof of Theorem~\ref{thm:verygeneralupperbound}. For each $i \in \{1,\dots,2k+2\}$, define the $r_k$-neighborhood of $A_i$ as
\begin{equation*}
A_i^{r_k} = \{x \in \Omega : d(x,A_i)<r_k\} \subset \R^n,
\end{equation*}
and the \emph{cylinder}
\begin{equation*}
\bar{A}_i^{r_k} = A_i^{r_k} \times \R^{m-n} \subset \R^{m}.
\end{equation*}
Since $\bar{A}_1^{r_k}, \dots, \bar{A}_{2k+2}^{r_k}$ are disjoint, there exist $k+1$ of them, say
$\bar{A}_1^{r_k}, \dots, \bar{A}_{k+1}^{r_k}$, such that, for $i \in \{1,\dots,k+1\}$, we have
\begin{equation*}
\vert M \cap \bar{A}_i^{r_k} \vert \leq \frac{\vert M \vert}{k+1}.
\end{equation*}
Similarly to \cite{cm}, for each $i \in \{1,\dots,k+1\}$, construct a test function $g_i$ with support in
$\bar{A}_i^{r_k}$ as follows. For each $x \in \bar{A}_i^{r_k}$,
\begin{equation*}
g_i(x) : = 1 - \frac{d(A_i \times \R^{m-n},x)}{r_k}.
\end{equation*}
Then $\vert \nabla g_i \vert^2 \leq \frac{1}{r_k^2}$ almost everywhere in $\bar{A}_i^{r_k}$. So we have that,
for each $i \in \{1,\dots,k+1\}$,
\begin{equation}\label{e1.2zf}
\int_{M} \vert \nabla g_i \vert^2 \, dV_M = \int_{M \cap \bar{A}_i^{r_k}} \vert \nabla g_i \vert^2 \, dV_M
\leq \frac{\vert M \cap \bar{A}_i^{r_k} \vert}{r_k^2} \leq \frac{\vert M \vert}{(k+1) r_k^2},
\end{equation}
and
\begin{align}\label{e1.2zg}
\int_{\Sigma} \vert g_i \vert^2 \, dV_\Sigma &= \int_{\Sigma \cap \bar{A}_i^{r_k}} \vert g_i \vert^2 \, dV_\Sigma
= \int_{\Sigma \cap A_i^{r_k}} \vert g_i \vert^2 \, dV_\Sigma \notag\\
&\geq \int_{\Sigma \cap A_i} dV_\Sigma = \mu_0(A_i) \geq \frac{\vert \Sigma_0 \vert}{4 C(n) (k+1)}.
\end{align}
Hence, by \eqref{e1.2zf} and \eqref{e1.2zg}, we have that the Rayleigh quotient of $g_i$ is
\begin{align*}
  R(g_i) &=
  \frac{\int_{M} \vert \nabla g_i \vert^2 \, dV_M}{\int_{\Sigma} \vert g_i \vert^2 \, dV_\Sigma}
  \leq \frac{\vert M \vert}{r_k^2} \frac{4 C(n)}{\vert \Sigma_0 \vert}
  < \frac{8 C(n)}{n \omega_n^{1/n} r_k^2} \frac{\vert M \vert }{\vert \Omega \vert^{(n-1)/n}} \notag\\
  &\leq C_2(n)
  \frac{\vert M \vert }{\vert  \Omega \vert^{(n+1)/n}} (k+1)^{2/(n-1)},
\end{align*}
for some constant $C_2(n)$. This implies that
\begin{equation}\label{ineq:finalstep4}
\sigma_k(\Sigma) \leq C_2(n)
\frac{\vert M \vert}{\vert \Omega \vert^{(n+1)/n}} (k+1)^{2/(n-1)}.
\end{equation}

\noindent
\textbf{Conclusion of the proof}. It
follows from Inequality \eqref{ineq:finalstep3} and Inequality
\eqref{ineq:finalstep4} that
\begin{align*}\sigma_k(\Sigma) &\leq \max\{C_1(n),C_2(n)\}\frac{\vert M
    \vert}{\vert \Omega \vert^{(n+1)/n}} (k+1)^{2/(n-1)}\\
  &\leq A(n)\frac{\vert M
    \vert}{\vert \Omega \vert^{(n+1)/n}} k^{2/(n-1)},
\end{align*}
where $A(n)=4\max\{C_1(n),C_2(n)\}$.

\end{proof}

\appendix
\section{Surfaces of revolution with one boundary component are Steklov isospectral}\label{section:isospectral}
Let $\Sp^1$ be the unit circle.
Let $M\subset\R^3$ be the surface of revolution with connected boundary
$\partial M=\Sp^1\subset\R^2\times\{0\}$.
\begin{prop}
  The surface $M$ is Steklov isospectral to the disk: for each $k\in\N$,
  $$\sigma_k(M)=\sigma_k(\mathbb{D}).$$
\end{prop}
In fact, the Dirichlet-to-Neumann map
$\Lambda:C^\infty(\Sp^1)\rightarrow C^\infty(\Sp^1)$ is the same
for both surfaces.

\begin{proof}
Let $\gamma:[0,L]\rightarrow\R^3$ be a unit-speed curve
$\gamma(r)=(h(r),0,z(r))$ such that $M$ is obtained by rotation around
the $z$-axis, with $r=0$
corresponding to the boundary.
In particular, the functions $h$ and $z$ satisfy $h'^2+z'^2=1$ from which it follows that
$$h(0)=1,\quad z(0)=0,\quad h(L)=0,\quad  z'(L)=0,\quad h'(L)=-1.$$
The corresponding metric on the cylinder $M=[0,L]\times \Sp^1$ is of the form
$$g=dr^2+h(r)^2 g_0$$
where $g_0=d\theta^2$ is the standard metric on the circle $\Sp^1$.
The Laplacian of $f\in C^\infty(M)$ is then given by
$$\Delta f=\frac{1}{h}\left\{(hf_r)_r+(h^{-1}f_\theta)_\theta\right\}.$$
Given $T\in\{\cos,\sin\}$ and $n\in\N$, let
$T_n(\theta)=T(n\theta)$. The Laplacian applied to the product
$f(r,\theta)=a(r)T_n(\theta)$ is
$$\Delta f=\frac{1}{h}\left\{(ha')'T_n+ah^{-1}((T_n)_{\theta\theta})\right\}
=\frac{1}{h}\left((ha')'-n^2ah^{-1}\right)T_n.$$
This implies that the function $f(r,\theta)=a(r)T_n(\theta)$ is harmonic
if and only if the function $a$ satisfies
\begin{gather}\label{equation:harmonic}
  (ha')'-n^2h^{-1}a=0.
\end{gather}
Note that in order for the function $f(r,\theta)=a(r)T_n(\theta)$ to
represent a continuous function on the surface of revolution $M$, it must
also satisfy the condition
$a(L)=0$.
Together with an initial value $a(0)$ this completely determines the function $a$.
\begin{lemma}\label{lemma:solution}
  Given $n\in\N$, and $c\in\R$ the unique function $a:[0,L]\rightarrow\R$ which satisfies
  \begin{gather*}
    \begin{cases}
      (ha')'-n^2h^{-1}a=0&\mbox{ in }(0,L);\\
      a(0)=c\quad \mbox{ and }\quad a(L)=0
    \end{cases}
  \end{gather*}
is given by
$a(r)=ce^{-n\tau(r)}$
where $\tau(r)=\int_0^r\frac{1}{h(s)}\,ds.$
\end{lemma}

It follows from Lemma \ref{lemma:solution} that the function
$f:M\rightarrow\R$ which is defined by
$f(r,\theta)=a(r)T_n(\theta)$ satisfies
$$\partial_{\nu}f=-a'(0)T_n(\theta)=nce^{-n\tau(0)}\tau'(0)T_n(\theta)=nf \quad \text{on $\partial M$}.$$
\end{proof}

\begin{proof}[Proof of Lemma \ref{lemma:solution}]
  In order to solve this equation,
  define $\tau:[0,L)\rightarrow\R$ by
    $$\tau(r)=\int_0^r\frac{1}{h(s)}\,ds.$$
  Let us first prove that $\tau(L)=+\infty$. Indeed, it follows from
  $h(L)=0$ and $h'(L)=-1$ that
  $$h(L-t)=h(L)+h'(L)(-t)+o(t)=t+o(t).$$
  Hence for small enough $t>0$, 
  $$\frac{1}{h(L-t)}=\frac{1}{t+o(t)}=\frac{1}{t}\left(\frac{1}{1+o(1)}\right)>\frac{1}{2t},$$
  which implies our claim that $\tau(L)=+\infty$.

  Therefore
  the function $\tau:[0,L)\rightarrow [0,\infty)$ is a bijection.
  Let $r:[0,\infty)\rightarrow [0,L)$ be its inverse.
  Let us prove that the function
  $\alpha:[0,\infty)\rightarrow\R$ defined by
  $\alpha(\tau)=a(r(\tau))$
  satisfies
  $$\alpha''=n^2\alpha\qquad\mbox{ on }\quad(0,\infty).$$
  Indeed, let's compute
  $$\alpha'(\tau)=a'(r(\tau))r'(\tau)=a'(r(\tau))h(r(\tau))=a'h(r(\tau)).$$
  Differentiating one more time leads to
  \begin{align*}
    \alpha''(\tau)&=(a'h(r(\tau)))'\\
    &=(a'h)'(r(\tau))r'(\tau)\\
    &=n^2h^{-1}a(r(\tau))h(r(\tau))\\
    &=n^2a(r(\tau))=n^2\alpha(\tau),
  \end{align*}
  as we claimed.
  It follows that there are two constants $A,B\in\R$ such that
  \begin{align*}
    a(r(\tau))&=\alpha(\tau)\\
    &=Ae^{n\tau}+Be^{-n\tau}.
  \end{align*}
  From the condition $a(L)=0$, it follows that
  $$\lim_{\tau\to\tau(L)=\infty}\alpha(\tau)=0.$$
  This amounts to $A=0$.
  The boundary condition $a(0)=c$ leads to $B=c$ and so
  $\alpha(\tau)=ce^{-n\tau}$.
\end{proof}

\medskip\noindent{\bf Acknowledgments.}
While a postdoctoral student at the Universit\'e de Neuch\^atel, KG was supported by the Swiss National Science Foundation grant no.\ 200021\_163228 entitled \emph{Geometric Spectral Theory}. AG acknowledges support from the NSERC Discovery Grants Program.

\bibliographystyle{plain}
\bibliography{biblioCGG}

\end{document}